\RequirePackage{ifthen} 

\def\mode{PreprintMode}

\def\JournalMode{
	\documentclass[graybox]{svmult}
}
\def\PreprintMode{
	\documentclass{article}
	\pdfoutput=1
}
\csname\mode\endcsname

\usepackage{graphicx}     
                            
\usepackage{float, mathtools}

\usepackage[latin1]{inputenc}
\usepackage[T1]{fontenc}
\usepackage{microtype} 

\usepackage{pgfplots, pgfplotstable}
\usepackage{tikz}
\usepgfplotslibrary{groupplots}

\usepackage{pgfkeys}
    \newenvironment{customlegend}[1][]{%
        \begingroup
        \csname pgfplots@init@cleared@structures\endcsname
        \pgfplotsset{#1}%
    }{%
        \csname pgfplots@createlegend\endcsname
        \endgroup
    }%
    \def\addlegendimage{
    \csname pgfplots@addlegendimage\endcsname}

\usepackage{pdfsync} 

\usepackage{lmodern}
\usepackage[hidelinks]{hyperref}
\usepackage{cite}

\usepackage{array,colortbl}
\usepackage{amsmath,amsfonts,amssymb,bm,mathtools,mathrsfs} 
\usepackage{dsfont}
\DeclareFontFamily{U}{mathx}{\hyphenchar\font45}
\DeclareFontShape{U}{mathx}{m}{n}{
      <5> <6> <7> <8> <9> <10>
      <10.95> <12> <14.4> <17.28> <20.74> <24.88>
      mathx10
      }{}
\DeclareSymbolFont{mathx}{U}{mathx}{m}{n}
\DeclareFontSubstitution{U}{mathx}{m}{n}
\DeclareMathAccent{\widecheck}{0}{mathx}{"71}

\definecolor{darkred}{RGB}{139,0,0}   
\definecolor{darkgreen}{RGB}{0,100,0} 
\definecolor{darkmagenta}{RGB}{180,0,180} 
\definecolor{darkblue}{RGB}{0,0,190}

\def\citep#1#2{\cite[{#1}]{#2}}

\usepackage[ruled, lined, linesnumbered, commentsnumbered, longend]{algorithm2e}

\newcommand{\bbN}{{\mathbb{N}}}

\newcommand{\bbR}{{\mathbb{R}}}

\newcommand{\bbU}{{\mathbb{U}}}

\newcommand{\bbZ}{{\mathbb{Z}}}

\newcommand{\N}{{\mathbb{N}}}

\newcommand{\Z}{{\mathbb{Z}}}

\DeclareSymbolFont{bbold}{U}{bbold}{m}{n}
\DeclareSymbolFontAlphabet{\mathbbold}{bbold}

\newcommand{\bsa}{{\boldsymbol{a}}}

\newcommand{\bsf}{{\boldsymbol{f}}}

\newcommand{\bsh}{{\boldsymbol{h}}}

\newcommand{\bsk}{{\boldsymbol{k}}}

\newcommand{\bsm}{{\boldsymbol{m}}}

\newcommand{\bsp}{{\boldsymbol{p}}}

\newcommand{\bst}{{\boldsymbol{t}}}

\newcommand{\bsx}{{\boldsymbol{x}}}
\newcommand{\bsy}{{\boldsymbol{y}}}
\newcommand{\bsz}{{\boldsymbol{z}}}

\newcommand{\bsell}{{\boldsymbol{\ell}}}
\newcommand{\bszero}{{\boldsymbol{0}}}
\newcommand{\bsone}{{\boldsymbol{1}}}

\newcommand{\bsgamma}{{\boldsymbol{\gamma}}}

\newcommand{\bsmu}{{\boldsymbol{\mu}}}

\newcommand{\calA}{{\mathcal{A}}}

\newcommand{\calK}{{\mathcal{K}}}

\newcommand{\calM}{{\mathcal{M}}}

\newcommand{\calO}{{\mathcal{O}}}
\newcommand{\calP}{{\mathcal{P}}}

\newcommand{\calS}{{\mathcal{S}}}

\newcommand{\fraku}{{\mathfrak{u}}}

\newcommand{\setu}{{\mathfrak{u}}}

\newcommand{\rme}{{\mathrm{e}}}

\newcommand{\rmi}{{\mathrm{i}}}

\newcommand{\tr}{{\rm tr}}

\newcommand{\rd}{\,\mathrm{d}}

\providecommand{\argmin}{\operatorname*{argmin}}

\usepackage{type1cm}        
\usepackage{makeidx}         

\usepackage{multicol}        
\usepackage[bottom]{footmisc}

\ifthenelse{\equal{\mode}{PreprintMode}}{
\usepackage{amsthm}

\newtheorem{theorem}{Theorem}[section]

\numberwithin{equation}{section}
}

\makeindex

\newcommand{\supp}{{\mathrm{supp}}}

\begin{document}

\ifthenelse{\equal{\mode}{JournalMode}}{
\title*{Comparison of Two Search Criteria for Lattice-based Kernel Approximation}
 \titlerunning{Comparison of $\calS_{n}^* (\bsz)$ and $\calP^*_n (\bsz)$ Used in Lattice-based Algorithms}
\author{Frances Y. Kuo, Weiwen Mo, Dirk Nuyens, Ian H. Sloan, 
Abirami Srikumar}
\institute{
   Frances Y.~Kuo,  Ian H.~Sloan and Abirami Srikumar
    \at School of Mathematics and Statistics, UNSW Sydney, Sydney NSW 2052, Australia \\
    \email{f.kuo@unsw.edu.au,
    i.sloan@unsw.edu.au,
    a.srikumar@student.unsw.edu.au}
    \and
    Weiwen Mo and Dirk Nuyens
    \at Department of Computer Science, KU Leuven, Leuven, Belgium \\
    \email{weiwen.mo@kuleuven.be,
    dirk.nuyens@kuleuven.be}
  }}
  {
  \title{Comparison of Two Search Criteria for Lattice-based Kernel Approximation}
\author{
    Frances Y.~Kuo\footnote{School of Mathematics and Statistics,
        University of New South Wales, Sydney NSW 2052, Australia,
        \texttt{(f.kuo|i.sloan)@unsw.edu.au|a.srikumar@student.unsw.edu.au}}
\and
	Weiwen Mo\footnote{Department of Computer Science, KU Leuven,
        Celestijnenlaan 200A, 3001 Leuven, Belgium,
        \texttt{(weiwen.mo|dirk.nuyens)@kuleuven.be}}
\and
    Dirk Nuyens\footnotemark[2]
\and 
     Ian H.~Sloan\footnotemark[1]
 \and 
    Abirami Srikumar\footnotemark[1]
    }
  }

\maketitle

\abstract{The kernel interpolant in a reproducing kernel Hilbert space is optimal in the worst-case sense among all approximations of a
function using the same set of function values. In
this paper, we compare two search criteria to construct lattice point sets
for use in lattice-based kernel approximation. The first candidate,
$\calP_n^*$, is based on the power function that appears in machine
learning literature. The second, $\calS_n^*$, is a search criterion used
for generating lattices for approximation using truncated Fourier series.
We find that the empirical difference in error between the lattices
constructed using $\calP_n^*$ and $\calS_n^*$ is marginal. The
criterion $\calS_n^*$ is preferred as it is computationally
more efficient and has a proven error bound.}

\section{Introduction} \label{MoSrsec:1}

Kernel interpolation seeks an approximation that interpolates a function~$f$
defined over $[0,1]^d$ at $n$ points (see \cite{KKKN20, ZKH09, ZLH06}). The approximation is formed using the reproducing kernel of a reproducing kernel Hilbert space $H$,
and
is of the form
\begin{align} \label{MoSreq:ker_int}
A_n^* (f) (\bsy) \coloneqq \sum_{k=0}^{n-1} a_k \, K(\bst_k , \bsy)\qquad\mbox{for}\quad \bsy \in [0,1]^d.
\end{align}
Here $K(\cdot,\cdot)$ is the reproducing kernel of $H$ and our $n$ distinct interpolation
points are given by $\bst_k \in [0,1]^d$ for $k=0,\ldots,n-1$. The
kernel interpolant is optimal in the worst-case sense among all approximations that use the same
function values of $f$ for which a proof can be
found in \cite{KKKN20}.

The quality of approximation depends on the choice of the $n$ interpolation
points, which leads us to ponder how we can
obtain a ``good'' set of points to reduce the approximation error. In this
paper, we will be considering \emph{lattice-based kernel approximation},
whereby the $n$ interpolation points form an $n$-point rank-$1$
lattice, i.e., a set of lattice points over $[0,1]^d$ characterised
by a generating vector $\bsz \in \bbU_{n}^{d}$, with
\begin{align}
\bst_k \coloneqq \left\{\frac{k \bsz}{n}\right\} \qquad\mbox{for}\quad k = 0,\ldots,n-1,
 \label{MoSreq:rank1}
\end{align}
where  $\bbU_n \coloneqq \{1 \le z \le n-1: \gcd(z,n) =1 \}$ and
$\{\,\cdot\,\}$ denotes taking the fractional part of each
component in a vector.

We use a \emph{component-by-component} (CBC)  algorithm to construct the
generating vector $\bsz$ to define a ``good'' lattice. A CBC algorithm
constructs $\bsz = (z_1,\ldots,z_d)$ by
selecting successive components $z_j$ from the set~$\bbU_n$ to minimise a
computable error criterion at each dimension, or to satisfy a certain
condition (in the case of ``reconstruction lattices''). It is known that
CBC construction of lattice generating vectors can ensure good error
bounds for integration and approximation in high dimensions, see e.g.,
\cite{CKNS20,CN08,DKS13,GIKV21,KMN22,KSW06,KV19,LM12,Nuy14}.

We consider the weighted Korobov space $H_{d,\alpha,\bsgamma}$ of
$d$-variate, one-periodic $L_2$ functions defined on $[0,1]^d$ with
absolutely converging Fourier series (see Section~\ref{MoSrsec:Prelim}).
Here $\alpha>1/2$ is known as the smoothness parameter. When $\alpha$
is an integer, functions in $H = H_{d,\alpha,\bsgamma}$ have square-integrable
mixed partial derivatives of at most order $\alpha$ in each coordinate.
Further, $\bsgamma := \{\gamma_\fraku\}_{\fraku\subset\N}$ are positive
weights quantifying the relative importance of different
subsets of variables.

The worst-case error for a given approximation algorithm $A_n$ with
respect to the $L_2$-norm is defined as
\begin{align*}
 e^{\rm wor} (A_n;L_2)
 \,\coloneqq\, \sup_{\|f\|_{d,\alpha,\bsgamma} \leq 1} \|f - A_n(f)\|_{L_2 ([0,1]^d)},
\end{align*}
where $\| \cdot \|_{d,\alpha,\bsgamma}$ denotes the Korobov space
norm (see \eqref{MoSreq:norm} below). It is difficult to obtain a
computable form for the worst-case error, hence an upper bound on the
worst-case error has been used as the search criterion for CBC
construction.

Kernel approximation has been applied to the interpolation of scattered
multivariate data by radial basis functions and is a recurrent topic in
machine learning and signal processing \cite{BBC20, SW06}. Given any
interpolation pointset $\Lambda = \{\bst_0,\ldots,\bst_{n-1}\}$ (not
necessarily a lattice), the \emph{power function} $P_\Lambda(\bsy)$ is
defined as the norm of the pointwise error functional \cite{DSW03, S95,
WS93}. When $\Lambda$ is a lattice pointset with generating vector
$\bsz$, the power function is thus exactly the worst-case pointwise error
of our lattice-based kernel approximation, $A^*_n$,
\begin{align} \label{MoSreq:powerfun}
P_{\Lambda}(\bsy) \,\coloneqq\, \sup_{\|f\|_{d,\alpha,\bsgamma}\leq 1} |f(\bsy)-A_n^*(f)(\bsy)|.
\end{align}
It follows easily that
\begin{align} \label{MoSreq:P_criteria}
e^{\rm wor}(A^*_n;L_2)  \leq  \bigg(  \int_{[0,1]^d} P_{\Lambda}(\bsy)^2  \rd \bsy \bigg)^{1/2} =: \calP_n^* (\bsz).
\end{align}
The quantity $\calP_n^* (\bsz)$ is a potential search criterion for
CBC construction. Explicit formulas for $P_{\Lambda}(\bsy)$ and
$\calP^*_n(\bsz)$ can be found in \eqref{MoSreq:xinorm1} and
\eqref{MoSreq:wcemtx} below.

A greedy data-independent method was proposed in \cite{DSW03}
whereby larger and larger point sets $\Lambda$ (not necessarily
lattices) are constructed by including the point which maximises the
power function constructed from the current data set. In their setting,
the rate of convergence of this algorithm depends poorly on dimension.

It appears that \cite{ZLH06} was the first to use lattice points as the
interpolation~set~$\Lambda$. It was shown in \cite[Theorem~3]{ZLH06}
that there exists a generating vector~$\bsz$ such that $\calP_n^* (\bsz) $
converges at the rate of $\calO(n^{-\alpha/2 + 1/4 + \delta})$ where $
\delta>0$. Both \cite{ZLH06} and the subsequent paper
\cite{ZKH09} on the $L_{\infty}$ error of kernel interpolation lack an
explicit CBC construction of a rank-$1$ lattice for kernel approximation.

By the optimality of kernel approximation, we have that
\begin{align*}
 e^{\rm wor} (A_n^*;L_2) \le  e^{\rm wor} (A_n;L_2)
\end{align*}
for any approximation $A_n$ using the same function values of $f$ at the
lattice pointset $\Lambda$. In \cite{CKNS21,KMN22} it was shown that a
truncated trigonometric polynomial approximation $A_n$ using lattice
points satisfies $e^{\rm wor} (A_n;L_2) \le \calS^*_n(\bsz)$, with
$\calS^*_n(\bsz)$ given by \eqref{MoSreq:S} below. Thus
\begin{align*}
e^{\rm wor} (A_n^*;L_2) \leq \min\{\calS^{*}_n(\bsz),\calP^*_n(\bsz)\},
\end{align*}
which offers $\calS^*_n(\bsz)$ as a second choice of search criterion for
the CBC construction for kernel approximation.

It was proved in \cite{CKNS21,KMN22} that a CBC construction based on
$\calS^*_n(\bsz)$ achieves a convergence rate of $\calO(n^{-\alpha/2 +
\delta})$ for $\delta>0$, which is also the best possible convergence rate
for lattice-based algorithms using a full rank-$1$ lattice (see
\cite{BKUV17, KKKN20}). It should be noted that lattice-based
approximation algorithms in general are not optimal, but are half of the
optimal convergence rate~$\calO(n^{-\alpha + \delta})$ in our setting (see \cite{DKU23, NSW04}). Algorithms based on
information from linear functionals or only function values (not
lattice-based) can achieve better rates, however, lattice-based algorithms
are easier and more efficient to implement.

Our investigation finds that using $\calS^*_n(\bsz)$ as the search
criterion is more efficient since a ``fast'' CBC algorithm can be used (see e.g., \cite{Nuy14}). As far as we know, no such fast algorithm
exists for the $\calP^*_n(\bsz)$ criterion. Further $\calS^*_n(\bsz)$ can
be computed accurately for large $n$ using double precision while
$\calP^*_n(\bsz)$ requires higher precision. We also find that the
difference in error measured by $\calP^*_n(\bsz)$ between the lattice
generated by minimising $\calP^*_n(\bsz)$ and the lattice generated by
minimising $\calS^*_n(\bsz)$ is marginal.

The structure of the paper is as follows.
Section~\ref{MoSrsec:Prelim} will detail some necessary background
required in Section~\ref{MoSrsec:P_criterion} for the derivation of the
kernel method upper bound and its computable form. In
Section~\ref{MoSrsec:CBCAlg}, an explicit CBC algorithm using $\calP_n^*
(\bsz)$ as the search criterion is proposed. Finally,
Section~\ref{MoSrsec:Num} provides a numerical comparison between
$\calP_n^* (\bsz)$ and $\calS_n^{*} (\bsz)$ using generating vectors
obtained from their respective CBC algorithms for different parameters.

\section{Preliminaries} \label{MoSrsec:Prelim}

\subsection{Weighted Korobov Spaces}

For $\alpha>\frac{1}{2}$ and positive weight parameters $\bsgamma
\coloneqq \{\gamma_{\setu} \}_{\setu \subset \bbN}$,  we consider the
Hilbert space $H_{d,\alpha,\bsgamma}$ of one-periodic $L_2$ functions
defined on $[0,1]^d$ with absolutely convergent Fourier series 
\begin{align*}
  f(\bsy) &\,=\, \sum_{\bsh\in\bbZ^d} \widehat{f} (\bsh) \, \rme^{2\pi \rmi \bsh\cdot\bsy}
  \quad \text{with} \quad
  \widehat{f} (\bsh) \,\coloneqq\, \int_{[0,1]^d} f(\bsy)\, \rme^{-2\pi \rmi \bsh\cdot\bsy}\,\rd\bsy,
\end{align*}
where $\bsh\cdot\bsy = \sum_{j=1}^d h_j y_j$ denotes the usual dot
product. The norm in $H_{d,\alpha,\bsgamma}$ is defined by
\begin{align}
  \|f\|_{d,\alpha, \bsgamma}^2
  & \coloneqq\, \sum_{\bsh\in\bbZ^d} \big\lvert \widehat{f} (\bsh) \big\rvert ^2\, r_{d,\alpha,\bsgamma}(\bsh) ,  \label{MoSreq:norm} \\
 r_{d,\alpha,\bsgamma} (\bsh) &\coloneqq \, \frac{1}{\gamma_{\supp (\bsh)}}  \prod_{j \in \supp (\bsh)} \lvert h_j \rvert^{2\alpha},  \notag
\end{align}
where $\supp (\bsh)\coloneqq \{1 \leq j \leq d: h_j \neq 0\}.$ The
parameter $\alpha$ characterizes the rate of decay of the Fourier
coefficients in the norm, and for integer $\alpha$ can be considered as a
smoothness parameter which indicates that $f$ has square-integrable mixed
partial derivatives of order $\alpha$ over all possible subsets of
variables. 

The inner product of $H_{d,\alpha,\bsgamma}$ is given by
\begin{align*}
  \left \langle f,g \right \rangle_{d, \alpha,\bsgamma}
  \, \coloneqq \,
  \sum_{\bsh \in \bbZ^d} \widehat{f} (\bsh) \, \overline{\widehat{g} (\bsh)} \, r_{d,\alpha,\bsgamma}(\bsh),
\end{align*}
and the norm is $\| \cdot \|_{d, \alpha,\bsgamma} = \langle \cdot,\cdot
\rangle_{d, \alpha,\bsgamma}^{1/2}$ which is consistent with
\eqref{MoSreq:norm}. Further, $H_{d,\alpha,\bsgamma}$ is a reproducing
kernel Hilbert space with reproducing kernel,
\begin{align*} 
K(\bsx,\bsy) = \sum_{\bsh \in \bbZ^d}
				\frac{ \rme^{2 \pi \rmi \bsh \cdot (\bsx - \bsy)} }{ r_{d,\alpha,\bsgamma}(\bsh)  },
\end{align*}
which satisfies (i) $K(\bsx, \bsy) = K(\bsy, \bsx)$ for all $\bsx, \bsy
\in [0,1]^d$; (ii) $K(\cdot, \bsy) \in H_{d,\alpha,\bsgamma}$ for all
$\bsy \in [0,1]^d$; (iii) $\left \langle f , K(\cdot, \bsy) \right
\rangle_{d,\alpha,\bsgamma} = f(\bsy)$ for all $f \in
H_{d,\alpha,\bsgamma}$ and all $\bsy \in [0,1]^d$. The last property is
known as the \emph{reproducing property}. It should be noted that
$r_{d,\alpha,\bsgamma}(\bsh)~=~r_{d,\alpha, \bsgamma} (-\bsh) $ and
therefore $K(\cdot,\cdot)$ takes only real values.

For integer $q > 1$, we have (see \cite[(24.8.3)]{DLMF}) 
\begin{align} \label{MoSreq:Bern}
\frac{ - (2 \pi \rmi)^{q}}{q!} B_{q}(y)
= \sum_{h \in \bbZ \backslash \{0\}} \frac{\rme^{2 \pi \rmi h y}}{h^{q}} \quad\mbox{for}\quad y \in [0,1].
\end{align}
Hence the kernel can be expressed in terms of periodic Bernoulli polynomials, 
\begin{align*}
K(\bsx,\bsy) =
\sum_{\setu \subseteq \{1:d\}} \gamma_{\setu} \prod_{j \in \setu}
		\left[ (-1)^{\alpha+1} \frac{(2 \pi)^{2 \alpha}}{(2 \alpha)!}
		B_{2 \alpha} \left( \left\{ x_j - y_j \right\} \right) \right],
\end{align*}
where $\{1:d\}$ is the set of integers from 1 to $d$ and and as before, the braces denote taking the fractional part of the input.

\subsection{The kernel interpolant} \label{MoSrsec:KerInt}

We approximate $f \in H_{d,\alpha,\bsgamma}$ by the kernel interpolant of
the form \eqref{MoSreq:ker_int} which interpolates $f$ at $n$ rank-$1$
lattice points given by~\eqref{MoSreq:rank1}, i.e.,
\begin{align} \label{MoSreq:interp1}
A_n^*(f) (\bst_{\ell}) = f(\bst_\ell) \quad \mbox{ for all } \ell= 0, \ldots, n-1.
\end{align}
The coefficients $a_k$, $k=0, \ldots, n-1$ are obtained by combining \eqref{MoSreq:ker_int} and \eqref{MoSreq:interp1} and solving the resulting linear system,
\begin{align} \label{MoSreq:a1}
f(\bst_{\ell}) = \sum_{k=0}^{n-1} a_k\, K(\bst_{k}, \bst_{\ell})
\quad \mbox{ for all } \ell= 0, \ldots, n-1.
\end{align}

To simplify our notation, we define the matrix
\begin{align}\label{MoSreq:K}
\calK  \,\coloneqq\, \begin{bmatrix} K(\bst_{k}, \bst_{\ell}) \end{bmatrix}_{\ell, \, k = 0, \ldots,n-1},
\end{align}
 and the vectors
\begin{align} \label{MoSreq:def_apk}
\bsa \coloneqq \begin{pmatrix} a_0 \\ a_1 \\ \vdots \\ a_{n-1} \end{pmatrix}, \quad
\bsf_{\Lambda} \coloneqq \begin{pmatrix} f(\bst_0) \\ f(\bst_1) \\ \vdots \\ f(\bst_{n-1}) \end{pmatrix}, \quad
\bsk_{\Lambda}(\bsy) \coloneqq \begin{pmatrix}K(\bst_0 , \bsy)\\ K(\bst_1 , \bsy) \\ \vdots \\ K(\bst_{n-1} , \bsy) \end{pmatrix}.
\end{align}
Note that the matrix $\calK$ is circulant and symmetric.

Then \eqref{MoSreq:a1} is equivalent to the following linear system
\begin{align}  \label{MoSreq:a_vec}
\bsf_{\Lambda} = \calK \, \bsa.
\end{align}
If $\calK$ has full rank, then the inverse $\calK^{-1}$ exists and
the solution $\bsa$ to \eqref{MoSreq:a_vec} is unique, i.e.,
\begin{align*} 
\bsa = \calK^{-1} \bsf_{\Lambda},
\end{align*}
and the inverse~$\calK^{-1}$ inherits the circulant structure and symmetry
of  matrix~$\calK$. Both $\bsa$ and $ \calK^{-1}$ can be obtained using
the Fast Fourier Transform (FFT).

An equivalent expression of $A_n^*(f) (\bsy)$ using the defined notation is,
\begin{align} \label{MoSreq:mu_exp}
A_n^*(f) (\bsy) 
= \bsa^{\top} \, \bsk_{\Lambda}(\bsy)
= \bsf_{\Lambda}^{\top} \, \calK^{-1} \, \bsk_{\Lambda}(\bsy).
\end{align}

\subsection[The S criterion]{The $\calS^*_n(\bsz)$ criterion} \label{MoSrsec:Fourier}
The truncated trigonometric polynomial approximation from
\cite{CKNS20, KMN22} is defined as follows. We first truncate the Fourier
expansion of $f$ to a finite index set $\calA_d(M) \,\coloneqq\,
\big\{\bsh\in\bbZ^d : r_{d,\alpha,\bsgamma}(\bsh) \le M \big\}
\subset\bbZ^d$ and then approximate Fourier coefficients,
$\widehat{f}(\bsh)$ for $\bsh\in \calA_d(M)$, using an $n$-point
rank-$1$ lattice rule, i.e.,
\begin{align*} 
 f(\bsy) \,\approx \, \sum_{\bsh\in\calA_d(M)} \left(
 \frac{1}{n} \sum_{k=0}^{n-1} f(\bst_k)\,\rme^{-2\pi \rmi \bsh \cdot \bst_k} \right)
 \rme^{2\pi \rmi \bsh \cdot \bsy}.
\end{align*}
With $M$ chosen to minimise the sum of the truncation and
approximation error bound, it was shown in \cite{CKNS20, KMN22} that an
upper bound on the worst-case error 
is 
\begin{align} \label{MoSreq:S}
\calS_n^* (\bsz) \coloneqq \sqrt{2} \left[ S_{n,d, \alpha, \bsgamma}(\bsz )   \right]^{1/4} = \calO(n^{-\alpha/2+\delta}),\qquad \delta>0,
\end{align}
with
\begin{align*}
S_{n,d, \alpha, \bsgamma}(\bsz )
= \sum_{\bsh\in\bbZ^d} \frac{1}{r_{d,\alpha,\bsgamma}(\bsh)}
  \sum_{\substack{\bsell\in\bbZ^d\setminus\{\bszero\} \\ \bsell\cdot\bsz\equiv_n 0}}
  \frac{1}{r_{d, \alpha, \bsgamma} (\bsh + \bsell)}.
\end{align*} 
The implied constant in \eqref{MoSreq:S} depends on $\alpha$ and the weight parameters $\bsgamma$. 
It is independent of dimension $d$ if $\bsgamma$ satisfies a certain condition.
The above convergence rate applies for both prime and composite $n$. For embedded rules, 
the convergence order is scaled by a logarithmic factor of $n$ (see \cite{KMN22}).

In the case of product weights, $\gamma_{\setu} = \prod_{j \in \setu}
\gamma_j$, a simple, computable expression for  $S_{n,d, \alpha,
\bsgamma}(\bsz )$ found in \cite{DKKS07} is 
\begin{align}\label{MoSreq:S_comp}
S_{n,d, \alpha, \bsgamma}(\bsz ) =
&-\prod_{j=1}^d (1+2\gamma_j^2 \zeta(4\alpha))
+ \frac{1}{n}\prod_{j=1}^d (1+2\gamma_j\zeta(2\alpha))^2\\
&\qquad\qquad + \frac{1}{n}\sum_{k=1}^{n-1} \prod_{j=1}^d
\left( 1+(-1)^{\alpha+1}\gamma_j\frac{(2\pi)^{2\alpha}}{(2\alpha)!}B_{2\alpha}\left(\left\{\frac{k z_j}{n}\right\}\right) \right)^2. \notag
\end{align}

Algorithm~\ref{MoSralg:CBC_four_fft} makes use of an alternative
formula for $S_{n,d, \alpha, \bsgamma}(\bsz)$ for product weights, found
in \cite{CKNS21}. Although this formula looks more complex, it is
mathematically equivalent to \eqref{MoSreq:S_comp} and is more accurate in
lower precision than the algorithm using \eqref{MoSreq:S_comp}. The
matrix-vector multiplications at line~$7$ can be computed using FFT after
reordering the rows and columns of the matrices $\Omega_n$ and $\Psi_n$
into circulant matrices. Thus, the cost of
Algorithm~\ref{MoSralg:CBC_four_fft} is $\calO(d \, n \log n)$ (see e.g., \cite{Nuy14, NC06a, NC06b}). For more details,
such as implementation for other types of weights and embedded lattice
sequences for approximation, the reader is referred to \cite{CKNS20,
CKNS21, KMN22, KSW06}.

\begin{algorithm}[t] \label{MoSralg:CBC_four_fft}
    \SetKwInOut{KwIn}{Input}
    \SetKwInOut{KwOut}{Output}
    \KwIn{The number of points $n \ge 2$;
    parameter $\alpha > 1/2$;
    weights $\{\gamma_j\}_{j\geq 1}$;
    maximum dimension $d$.}
    \KwOut{Generating vector $\bsz$; value of $\calS_n^* (\bsz)$.}

    	Set $S = 0$\\
   	 $\Omega_n = \begin{bmatrix} \omega(z,k) \end{bmatrix}_{z\in \bbU_n, k = 0, \ldots,n-1}  $ with
   		$ \omega(z,k) = (-1)^{\alpha+1} \frac{(2\pi)^{2\alpha}}{(2\alpha)!}B_{2\alpha}\left(\left\{\frac{k \, z}{n}\right\}\right)$ \\
	$\Psi_n = \begin{bmatrix} \left(\omega(z,k)\right)^2 - 2\zeta(2\alpha) \end{bmatrix}_{z\in \bbU_n, k = 0, \ldots,n-1}  $\\
	
	$\bsp_0 = \bsone \cdot \prod_{j=2}^d (1+\zeta(2\alpha)\gamma_j^2) \qquad \qquad$ \tcp{Vector of size $n$}

    \tcc{CBC construction of generating vector $\bsz$}
    \For{$s \leftarrow 1$ \KwTo $d$}{
     $W_{d,s} = \frac{1}{n} \Psi_n\,\gamma_s^2 \, \bsp_{s-1} + \frac{2}{n} \Omega_n \,\gamma_s\, \bsp_{s-1} \quad$  \tcp{$W_{d,s}$ is a vector of size $\lvert \bbU_n \rvert$}
     find $\min W_{d,s}$ and set $z_s = \argmin W_{d,s}$, then $\bsz = [\bsz,z_s]$, $S = S + \min W_{d,s}$\\
     $\bsp_s = (1+\gamma_s\Omega_n(z_s,:))^2 \, .\!* \, \bsp_{s-1}/(1+\zeta(2\alpha)\gamma_{s+1}^2)$
    }
    $\calS_n^* (\bsz) = \sqrt{2} \, S^{1/4}$

    \caption{Fast CBC algorithm using $\calS^*_n(\bsz)$ }
\end{algorithm}

\section[The P criterion]{Formulation of search criterion $\calP^*_n (\bsz)$} \label{MoSrsec:P_criterion}

We begin with a derivation of the upper bound, $\calP_n^* (\bsz)$, on the worst-case approximation error for the kernel interpolant method and an explicit expression for this upper bound. 

\begin{theorem} \label{MoSrthm: intP}
The quantity $\calP_n^*(\bsz)$ defined in \eqref{MoSreq:P_criteria} can be written as
\begin{align}  \label{MoSreq:wce2}
\calP_n^* (\bsz) =  \bigg( K(\bsy,\bsy)
		- \int_{[0,1]^d} \bsk_{\Lambda}(\bsy)^{\top} \calK^{-1} \bsk_{\Lambda}(\bsy) \, \rd \bsy \bigg)^{1/2},
\end{align}
where $K(\cdot,\cdot)$ is the reproducing kernel of $H_{d,\alpha,\bsgamma}$, $\bsk_{\Lambda}$ is defined in \eqref{MoSreq:def_apk} and $\calK^{-1}$ is
the inverse of matrix $\calK$ defined in \eqref{MoSreq:K}. This holds for any reproducing kernel.
\end{theorem}

\begin{proof}
Fix $\bsy\in[0,1]^d$. Define the vector $\bsmu_{\Lambda}(\bsy)\coloneqq \calK^{-1}  \bsk_{\Lambda}(\bsy)$ with the form
$$
\bsmu_{\Lambda}(\bsy) =  \begin{pmatrix} \mu_0(\bsy) & \mu_1(\bsy) & \ldots & \mu_{n-1}(\bsy) \end{pmatrix}^{\top}.
$$
Using \eqref{MoSreq:mu_exp} and the reproducing property of the kernel, we have
\begin{align*}
f(\bsy) - A_n^*(f) (\bsy)
&= f(\bsy) -  \bsf_{\Lambda}^{\top} \bsmu_{\Lambda}(\bsy)
= f(\bsy) - \sum_{k=0}^{n-1} f(\bst_k) \mu_k(\bsy) \\
&= \bigg\langle f ,  K(\cdot, \bsy) -  \sum_{k=0}^{n-1}  K(\cdot, \bst_k) \mu_k(\bsy) \bigg\rangle_{d,\alpha,\bsgamma}.
\end{align*}

Applying the Cauchy--Schwarz inequality, we obtain, 
\begin{align*} 
\lvert f(\bsy) - A_n^*(f) (\bsy)  \rvert
&\le \| f \|_{d,\alpha,\bsgamma } \big \|  K(\cdot, \bsy) -  \sum_{k=0}^{n-1}  K(\cdot, \bst_k) \mu_k(\bsy)  \big\|_{d,\alpha,\bsgamma},
\end{align*}
where equality is attained at $\bsy$ when $f$ and $K(\cdot, \bsy) -  \sum_{k=0}^{n-1}  K(\cdot, \bst_k) \mu_k(\bsy) $ are linearly dependent.
So the power function defined by \eqref{MoSreq:powerfun} is actually 
the norm of the error functional on $H_{d,\alpha,\bsgamma}$ evaluated at $\bsy$, with an equivalent expression given as follows
\begin{align*}
 P_{\Lambda}(\bsy) = \big \|  K(\cdot, \bsy) -  \sum_{k=0}^{n-1}  K(\cdot, \bst_k) \mu_k(\bsy)  \big\|_{d,\alpha,\bsgamma}.
\end{align*}

Now we have
\begin{align} \label{MoSreq:xinorm1}
&P_{\Lambda}(\bsy)^2
= \bigg\langle  K(\cdot, \bsy) -  \sum_{k=0}^{n-1}  K(\cdot, \bst_k) \mu_k(\bsy),
		 K(\cdot, \bsy) -  \sum_{\ell=0}^{n-1}  K(\cdot, \bst_{\ell}) \mu_{\ell}(\bsy) \bigg\rangle_{d,\alpha,\bsgamma} \notag \\	
&= K(\bsy,\bsy) - 2  \sum_{k=0}^{n-1} \mu_{k}(\bsy)  K(\bst_k,\bsy)
		+ \sum_{k=0}^{n-1}  \sum_{\ell=0}^{n-1}   \mu_{k}(\bsy)	 \mu_{\ell}(\bsy) K(\bst_k,\bst_{\ell}).
\end{align} 

Since $\bsk_{\Lambda}(\bsy) = \calK \bsmu_{\Lambda}(\bsy)$, we have
\begin{align} \label{MoSreq:mu_k}
K(\bst_{\ell}, \bsy) = \sum_{k=0}^{n-1} \mu_k(\bsy) K(\bst_{\ell}, \bst_{k})
 \quad \mbox{for } \ell = 0, 1, \ldots, n-1.
\end{align}
We continue forward by substituting \eqref{MoSreq:mu_k} into \eqref{MoSreq:xinorm1} and simplifying,
\begin{align*} 
P_{\Lambda}(\bsy)^2
&=  K(\bsy,\bsy) - \sum_{k=0}^{n-1} \mu_{k}(\bsy)  K(\bst_k,\bsy)
=  K(\bsy,\bsy) - \bsk_{\Lambda}(\bsy)^{\top} \calK^{-1} \bsk_{\Lambda}(\bsy),
\end{align*}
as required.
\end{proof}

The following theorem expands on the previous theorem by providing a computable expression for $\calP_n^* (\bsz)$ for $\alpha >1/2$ and some given weight parameters $\bsgamma$. 
Note expression \eqref{MoSreq:wcemtx} has been derived in
\textup{\cite[Lemma 2]{ZLH06}}. 

\begin{theorem}
Given $\alpha > 1/2$ and weight parameters $\bsgamma\coloneqq \{\gamma_\fraku\}_{\fraku\subset\N}$, an equivalent expression for $\calP^*_n (\bsz)$ is given by
\begin{align} \label{MoSreq:wcemtx}
\calP^*_n (\bsz)
 = \bigg( \sum_{\setu \subseteq \{1:d\}}  \gamma_{\setu}  [2 \zeta(2\alpha)]^{|\setu|} - \tr(\calK^{-1} \calM) \bigg)^{1/2},
\end{align}
where $\zeta(\cdot)$ is the Riemann zeta function, matrix~$\calK$ is
defined in \eqref{MoSreq:K} and the elements of the symmetric and
circulant matrix $\calM \in \bbR^{n \times n}$ are given by
\begin{align} \label{MoSreq:exp_mlk}
\calM_{\ell,k} &\coloneqq  \int_{[0,1]^d} K(\bst_\ell, \bsy)\,  K(\bst_k, \bsy)  \, \rd \bsy \\
&\,= 
\sum_{\setu \subseteq
\{1:d\}} \gamma_{\setu}^2 \prod_{j \in \setu}
		\left[ - \frac{(2 \pi)^{4 \alpha}}{(4 \alpha)!}
		B_{4 \alpha} \left( \left\{\frac{(\ell - k)z_j }{n}\right\}
\right) \right]. \nonumber
\end{align}
\end{theorem}
\begin{proof}
From \eqref{MoSreq:def_apk} and the definition of matrix~$\calM$, we can write the integral in \eqref{MoSreq:wce2} as
\begin{multline}
 \int_{[0,1]^d} \bsk_{\Lambda}(\bsy)^{\top} \calK^{-1} \bsk_{\Lambda}(\bsy) \, \rd \bsy
 = \sum_{\ell=0}^{n-1}  \sum_{k=0}^{n-1}(\calK^{-1})_{\ell,k} \int_{[0,1]^d} K(\bst_\ell, \bsy) K(\bst_k, \bsy)  \, \rd \bsy	\notag\\
 =  \sum_{\ell=0}^{n-1}  \sum_{k=0}^{n-1} (\calK^{-1})_{\ell,k} \, \calM_{\ell,k}  = \tr(\calK^{-1} \calM),
\end{multline}
where we used the obvious symmetry of matrix $\calM$.
We also have that
$$K(\bsy, \bsy) = \sum_{\bsh\in \Z^d} \frac{1}{r_{d,\alpha,\bsgamma}(\bsh)}
= \sum_{\setu\subseteq \{1:d\}}\gamma_\setu \prod_{j\in\setu}\sum_{h\in\Z\backslash\{0\}}\frac{1}{|h|^{2\alpha}}
=\sum_{\setu \subseteq \{1:d\}}  \gamma_{\setu}  [2 \zeta(2\alpha)]^{|\setu|},$$
which can be combined with the above to achieve \eqref{MoSreq:wcemtx}.

An explicit expression for $\calM_{\ell,k}$ can be obtained as follows
\begin{align} \label{MoSreq:Mcirc}
\calM_{\ell,k}
&=  \int_{[0,1]^d}
			 \bigg( \sum_{\bsh \in \bbZ^d} \frac{\rme^{2 \pi \rmi \bsh \cdot (\bst_{\ell}-\bsy)}}{ r_{d,\alpha,\bsgamma} (\bsh) } \bigg)
			  \bigg( \sum_{\bsh' \in \bbZ^d} \frac{\rme^{2 \pi \rmi \bsh' \cdot (\bst_k-\bsy)}}{ r_{d,\alpha,\bsgamma} (\bsh') } \bigg)
			   \, \rd \bsy  \notag\\
&= 	\sum_{\bsh \in \bbZ^d} \sum_{\bsh' \in \bbZ^d}   \frac{\rme^{2 \pi \rmi \bsh \cdot \bst_{\ell}}}{ r_{d,\alpha,\bsgamma} (\bsh) }
		 \frac{\rme^{2 \pi \rmi \bsh' \cdot \bst_{k}}}{ r_{d,\alpha,\bsgamma} (\bsh') }
		\underbrace{ \int_{[0,1]^d} \rme^{- 2 \pi \rmi (\bsh+\bsh') \cdot \bsy}  \, \rd \bsy}_{=\delta_{\bsh,- \bsh'}}  \notag\\
&=  \sum_{\bsh \in \bbZ^d} \frac{1}{[ r_{d,\alpha,\bsgamma} (\bsh)]^2} \rme^{2 \pi \rmi \bsh \cdot \bsz (\ell- k)/n}.
\end{align}
Simplifying the expression above further results in 
\begin{align*}
\calM_{\ell,k}
&= \!\!\! \sum_{\bsh \in \bbZ^d} \gamma_{\supp(\bsh)}^2 \!\!\! \prod_{j \in \supp(\bsh)} \!\!\!\! \frac{ \rme^{2 \pi \rmi h_j z_j (\ell- k)/n}}{\lvert  h_j\rvert^{4 \alpha}}
=\!\!\! \sum_{\setu \subseteq \{1:d\}} \gamma_{\setu}^2 \prod_{j \in \setu}  \sum_{h \in \bbZ \backslash \{0\}}
		\!\!\!\! \frac{ \rme^{2 \pi \rmi h z_j (\ell- k)/n}}{\lvert  h\rvert^{4 \alpha}} \notag \\
&= \sum_{\setu \subseteq \{1:d\}} \gamma_{\setu}^2 \prod_{j \in \setu}
		\left[ - \frac{(2 \pi)^{4 \alpha}}{(4 \alpha)!}
		B_{4 \alpha} \left( \left\{\frac{(\ell - k)z_j }{n}\right\} \right) \right].
\end{align*}
where for the last equality we used \eqref{MoSreq:Bern}.

We see from \eqref{MoSreq:Mcirc} that each column of $\calM$ is just the previous column shifted down by one element and hence the matrix $\calM$ is circulant.
\end{proof}

As noted in \cite{ZLH06}, evaluating $\left[\calP_n^* (\bsz)\right]^2$ by
computing the integral~\eqref{MoSreq:exp_mlk} causes catastrophic
round-off errors when $\alpha >1$ and $d>2$. In \cite[Theorem~1]{ZLH06},
an alternative form of $\calP_n^* (\bsz)$ is proposed. The authors
approximate $\calP_n^* (\bsz)$ by truncating the infinite sums involved
with truncation error of $\calO(n^{-2\alpha})$. Instead, we will evaluate
$\calP_n^* (\bsz)$ in arbitrary precision using
Algorithm~\ref{MoSralg:CBC_ker_fft} below.

\section[CBC with P criterion]{CBC construction based on $\calP^*_n (\bsz)$} \label{MoSrsec:CBCAlg}

\begin{algorithm}[t] \label{MoSralg:CBC_ker_fft}
    \SetKwInOut{KwIn}{Input}
    \SetKwInOut{KwOut}{Output}

    \KwIn{The number of points $n \ge 2$;
    parameter $\alpha > 1/2$;
    weights $\{\gamma_j\}_{j\geq 1}$;
    maximum dimension $d$.}
    \KwOut{Generating vector $\bsz$; value of $\calP_n^* (\bsz)$.}

   \tcc{Initialization: first column of $\calK$ and $\calM$ at dimension $1$}
    Set $z_1 = 1$

    $\bsk_{1} = \begin{bmatrix} k_{1}(\ell) \end{bmatrix}_{\ell = 0}^{n-1}, $ \, \,
    $ k_{1}(\ell) =1 +  (-1)^{\alpha + 1} \frac{(2 \pi)^{2\alpha}}{(2\alpha)!}  \gamma_{1}  B_{2\alpha} \left( \left\{\frac{\ell }{n}\right\} \right)$

    $\bsm_{1} = \begin{bmatrix} m_{1}(\ell) \end{bmatrix}_{\ell = 0}^{n-1}, $ \, \,
    $ m_{1}(\ell) =1 -
		 \frac{(2 \pi)^{4 \alpha}}{(4 \alpha)!}  \gamma_{1}^2
		B_{4 \alpha} \left( \left\{\frac{\ell }{n}\right\} \right) $

    \tcc{CBC construction of generating vector $\bsz$}
    \For{$s \leftarrow 2$ \KwTo $d$}{
    with $z_1$, \ldots, $z_{s-1}$ fixed,  $\bsz = [ z_1, \ldots, z_{s-1} ]$
    
    \For{$z$ \rm{in} $\bbU_n = \{1 \le z \le n-1: \gcd(z,n) =1 \}$}{

    $\bsk_{s} = \begin{bmatrix} k_{s}(\ell) \end{bmatrix}_{\ell = 0}^{n-1}, $ \, \,
    $ k_{s}(\ell) = k_{s-1}(\ell)  \big[ 1 +   (-1)^{\alpha + 1} \frac{(2 \pi)^{2\alpha}}{(2\alpha)!} \gamma_s B_{2\alpha} \left( \left\{\frac{\ell z}{n}\right\} \right)  \big]$

    $\bsm_{s} = \begin{bmatrix} m_{s}(\ell) \end{bmatrix}_{\ell = 0}^{n-1}, $ \, \,
    $ m_{s}(\ell) = m_{s-1}(\ell)   \big[1 -
		 \frac{(2 \pi)^{4 \alpha}}{(4 \alpha)!} \gamma_{s}^2
		B_{4 \alpha} \left( \left\{\frac{\ell z }{n}\right\} \right)  \big]$
		
 	$\widehat{\bsk}_s = {\rm FFT}(\bsk_{s})$
 	
 	$\widehat{\bsm}_s = {\rm FFT}(\bsm_{s})$

	$T_s(z) = {\rm sum}(\widehat{\bsm}_s \, ./  \, \widehat{\bsk}_s )$
    }

    find $\max_{z \in \bbU_n} T_s(z)$ and set $z_s$ to be the maximiser, so $\bsz = [ \bsz , z_s]$
    }
	$ \calP_n^* (\bsz) = \left[\prod_{j=1}^d (1 + \gamma_j [2 \zeta(2 \alpha)] )  - T_d(z_d)\right]^{1/2}$
    \caption{CBC algorithm using $\calP_{n}^*(\bsz)$ as error criterion}
\end{algorithm}

In this section, we consider product weights, i.e., $\gamma_{\setu} =
\prod_{j \in \setu} \gamma_j$, for $\setu \subseteq \{1:d\}$. Since
$\sum_{\setu \subseteq\{1:d\}} \prod_{j \in \setu} a_j=\prod_{j=1}^d (1 +
a_j) $, the elements of matrices~$\calK$ and $\calM$ become
\begin{align*}
\calK_{\ell,k} 
= \prod_{j=1}^d \left( 1 + (-1)^{\alpha+1} \frac{(2 \pi)^{2 \alpha}}{(2 \alpha)!} \gamma_j
		B_{2 \alpha} \left( \left\{ \frac{(\ell - k)z_j}{n} \right\} \right) \right),
\end{align*}
and
\begin{align*}
\calM_{\ell,k} &= \prod_{j=1}^d \left( 1  - \frac{(2 \pi)^{4 \alpha}}{(4 \alpha)!} \gamma_j^2
		B_{4 \alpha} \left( \left\{\frac{(\ell - k)z_j }{n}\right\} \right)\right) .
\end{align*}

As $\calK$ and $\calM$ are circulant, the eigenvalues of both matrices can be obtained by applying the discrete Fourier transform (DFT) to their first columns, i.e.,
\begin{align*}
\widehat{k}_{\ell} = \sum_{j=0}^{n-1} \calK_{j,0} \, \rme^{- 2 \pi \rmi \frac{\ell j}{n}} \, \text{ and } \,
\widehat{m}_{\ell} =  \sum_{j=0}^{n-1} \calM_{j,0} \, \rme^{- 2 \pi \rmi \frac{\ell j}{n}}
 \,\mbox{ for } \, \ell = 0,1, \ldots, n-1.
\end{align*}
The matrix $\calK^{-1} \calM$ is also a circulant matrix with eigenvalues~$\widehat{m}_{\ell}/\widehat{k}_{\ell}$ where $\ell = 0, \dots, n-1$.
Thus, $\calP_{n}^*(\bsz)$, given by \eqref{MoSreq:wcemtx}, can be computed using,
\begin{align*}
\calP_{n}^*(\bsz) = \bigg[ \prod_{j=1}^d (1 + \gamma_j [2 \zeta(2 \alpha)]) - \sum_{\ell = 0}^{n-1} \frac{\widehat{m}_{\ell}}{\widehat{k}_{\ell}} \bigg]^{1/2},
 \end{align*}
 leading to CBC Algorithm \ref{MoSralg:CBC_ker_fft} with computational cost $\calO(n^2 \, d \log n)$.

\section{Numerical results} \label{MoSrsec:Num}
We finally present the results of numerical experiments using the $\calP_n^*(\bsz)$ and $\calS_n^{*}(\bsz)$ criteria. We plot the two upper bounds for different choices of weight parameters of the product form~$\gamma_j >0$ for both $\alpha =1$ and $\alpha =2$. 
The weight parameters include a scaling factor of  $\pi^{-2\alpha}$ so that the computable expressions for both criteria only consist of rational numbers. This allowed us to easily verify numerical results by calculating test cases analytically.
 Note that the generating vectors, $\bsz_{\rm cbc}^{\calS}$ and $\bsz_{\rm cbc}^{\calP}$,
are constructed by each individual CBC algorithm for each $n$.

\pgfplotsset{ tick label style={font=\small}, label style={font=\small},
legend style={font=\scriptsize}, every axis/.append style={ line
width=0.8 pt, tick style={line width=0.3pt}} }

\begin{figure}[t]
    \centering
    \begin{tikzpicture}

    \begin{customlegend}[legend columns=4, legend style={at={(4.2,-6)},anchor=north, align=center},,
            legend entries={$\calS_n^{*} (\bsz_{\rm cbc}^{\calS})$ ,
                            $\calP_n^* (\bsz_{\rm cbc}^{\calP})$,
                            }]
        \addlegendimage{color= blue, mark=o, dashed, mark options=solid, line width=0.8 pt,   line legend}
            \addlegendimage{color= magenta, mark=+, line width=0.8 pt,   line legend} 
    \end{customlegend}
    
    \begin{groupplot}[
            group style={
        group name=my plots, 
        group size=2 by 2,
        x descriptions at=edge bottom,
        ylabels at=edge left,
        },
                footnotesize,
                width=6cm,
                height=5.5cm,
                xlabel= {$n$},
                ylabel= {values of upper bounds},
            ]
       
       \nextgroupplot[xmode = log, ymode=log, grid = major, title = { $\gamma_j = j^{-3 \alpha} / \pi^{2 \alpha}$}]

	\addplot[color=blue,mark=o, dashed, mark options=solid] table{Num/Num_S_zScbc_2_10_invj3.txt};
	
	\pgfplotstableread{Num/Num_S_zScbc_2_10_invj3.txt}\tableA
	\pgfplotstablecreatecol[linear regression]{regression}{\tableA}
	\xdef\slopeA{\pgfplotstableregressiona}

	 \node[pin={[pin distance=0.1cm] -90:{\scriptsize $\textcolor{blue}{\pgfmathprintnumber[fixed, precision=3]{\slopeA}}$}},fill=blue,circle,scale=0.1,]
	(aninnernode) at (8,-3.2) {};

	\addplot[color=magenta,mark=+,]  table {Num/Num_P_zPcbc_2_10_invj3.txt};
	
	\pgfplotstableread{Num/Num_P_zPcbc_2_10_invj3.txt}\tableB
	\pgfplotstablecreatecol[linear regression]{regression}{\tableB}
	\xdef\slopeB{\pgfplotstableregressiona}
	\node[ pin={[pin distance=0.1cm] 90:{\scriptsize $\textcolor{magenta}{\pgfmathprintnumber[fixed, precision=3]{\slopeB}}$}}, fill=magenta,circle,scale=0.1,]
	(aninnernode) at (9.2,-2.8) {};
	\node[ pin={ [pin distance=0.1cm] -90:{\scriptsize $\alpha = 1$ }}, fill=blue,circle,scale=0.1,] (aninnernode) at (9.5,-4.5) {};

	\addplot[color=blue, mark=o, dashed, mark options=solid]  table{Num/Num_S_zScbc_4_10_invj6.txt};
	
	\pgfplotstableread{Num/Num_S_zScbc_4_10_invj6.txt}\tableC
	\pgfplotstablecreatecol[linear regression]{regression}{\tableC}
	\xdef\slopeC{\pgfplotstableregressiona}
	\node[ pin={ [pin distance=0.1cm] 90:{\scriptsize $\textcolor{blue}{\pgfmathprintnumber[fixed, precision=3]{\slopeC}}$}}, fill=blue,circle,scale=0.1,]
	(aninnernode) at (8,-8.4) {};

	\addplot[color=magenta,mark=+,]  table{Num/Num_P_zPcbc_4_10_invj6.txt};
	
	\pgfplotstableread{Num/Num_P_zPcbc_4_10_invj6.txt}\tableD
		\pgfplotstablecreatecol[linear regression]{regression}{\tableD}
	\xdef\slopeD{\pgfplotstableregressiona}
	\node[ pin={ [pin distance=0.1cm] -90:{\scriptsize $\textcolor{magenta}{\pgfmathprintnumber[fixed, precision=3]{\slopeD}}$}}, fill=magenta,circle,scale=0.1,]
	(aninnernode) at (9.2,-10.3) {};
	\node[ pin={ [pin distance=0.1cm] 90:{\scriptsize $\alpha = 2$ }}, fill=blue,circle,scale=0.1,] (aninnernode) at (7,-7) {};

	\nextgroupplot[xmode = log, ymode=log, grid = major, title = { $\gamma_j =  j^{-2}/\pi^{2 \alpha}$}]
	
	\addplot[color=blue,mark=o,  dashed, mark options=solid]  table {Num/Num_S_zScbc_2_10_invj2.txt};
	\pgfplotstableread{Num/Num_S_zScbc_2_10_invj2.txt}\tableE
	\pgfplotstablecreatecol[linear regression]{regression}{\tableE}
	\xdef\slopeE{\pgfplotstableregressiona}
	 \node[pin={[pin distance=0.1cm] -90:{\scriptsize $\textcolor{blue}{\pgfmathprintnumber[fixed, precision=3]{\slopeE}}$}},fill=blue,circle,scale=0.1,]
	(aninnernode) at (8,-2.7) {};

	\addplot[color=magenta,mark=+,]  table {Num/Num_P_zPcbc_2_10_invj2.txt};
	\pgfplotstableread{Num/Num_P_zPcbc_2_10_invj2.txt}\tableF
	\pgfplotstablecreatecol[linear regression]{regression}{\tableF}
	\xdef\slopeF{\pgfplotstableregressiona}
	 \node[pin={[pin distance=0.1cm] 90:{\scriptsize $\textcolor{magenta}{\pgfmathprintnumber[fixed, precision=3]{\slopeF}}$}},fill=magenta,circle,scale=0.1,]
	(aninnernode) at (9.2,-2.0) {};
	\node[ pin={ [pin distance=0.1cm] -90:{\scriptsize $\alpha = 1$ }}, fill=blue,circle,scale=0.1,] (aninnernode) at (9.5,-3.5) {};

	\addplot[color=blue, mark=o,  dashed, mark options=solid]  table{Num/Num_S_zScbc_4_10_invj2.txt};
	\pgfplotstableread{Num/Num_S_zScbc_4_10_invj2.txt}\tableG
	\pgfplotstablecreatecol[linear regression]{regression}{\tableG}
	\xdef\slopeG{\pgfplotstableregressiona}
	\node[pin={[pin distance=0.1cm] -90:{\scriptsize $\textcolor{blue}{\pgfmathprintnumber[fixed, precision=3]{\slopeG}}$}},fill=blue,circle,scale=0.1,]
	(aninnernode) at (8,-6.5) {};

	\addplot[color=magenta,mark=+,]  table{Num/Num_P_zPcbc_4_10_invj2.txt};
	\pgfplotstableread{Num/Num_P_zPcbc_4_10_invj2.txt}\tableH
	\pgfplotstablecreatecol[linear regression]{regression}{\tableH}
	\xdef\slopeH{\pgfplotstableregressiona}
	 \node[pin={[pin distance=0.1cm] 90:{\scriptsize $\textcolor{magenta}{\pgfmathprintnumber[fixed, precision=3]{\slopeH}}$}},fill=magenta,circle,scale=0.1,]
	(aninnernode) at (9.2,-7) {};
	\node[ pin={ [pin distance=0.1cm] 90:{\scriptsize $\alpha = 2$ }}, fill=blue,circle,scale=0.1,] (aninnernode) at (7,-5.3) {};

	\nextgroupplot[xmode = log, ymode=log, grid = major, title = { $\gamma_j = 0.9^{j-1}/\pi^{2 \alpha}$}]
	
	\addplot[color=blue,mark=o,  dashed, mark options=solid]  table {Num/Num_S_zScbc_2_10_expj9.txt};
	\pgfplotstableread{Num/Num_S_zScbc_2_10_expj9.txt}\tableI
	\pgfplotstablecreatecol[linear regression]{regression}{\tableI}
	\xdef\slopeI{\pgfplotstableregressiona}
	\node[pin={[pin distance=0.1cm] -90:{\scriptsize $\textcolor{blue}{\pgfmathprintnumber[fixed, precision=3]{\slopeI}}$}},fill=blue,circle,scale=0.1,]
	(aninnernode) at (8,-0.85) {};

	\addplot[color=magenta,mark=+,]  table {Num/Num_P_zPcbc_2_10_expj9.txt};
	\pgfplotstableread{Num/Num_P_zPcbc_2_10_expj9.txt}\tableJ
	\pgfplotstablecreatecol[linear regression]{regression}{\tableJ}
	\xdef\slopeJ{\pgfplotstableregressiona}
	\node[pin={[pin distance=0.1cm] -90:{\scriptsize $\textcolor{magenta}{\pgfmathprintnumber[fixed, precision=3]{\slopeJ}}$}},fill=magenta,circle,scale=0.1,]
	(aninnernode) at (9.2, 0.34) {};
	\node[ pin={ [pin distance=0.1cm] -90:{\scriptsize $\alpha = 1$ }}, fill=blue,circle,scale=0.1,] (aninnernode) at (9.5,-1.5) {};

	\addplot[color=blue, mark=o,  dashed, mark options=solid]  table{Num/Num_S_zScbc_4_10_expj9.txt};
	\pgfplotstableread{Num/Num_S_zScbc_4_10_expj9.txt}\tableK
	\pgfplotstablecreatecol[linear regression]{regression}{\tableK}
	\xdef\slopeK{\pgfplotstableregressiona}
	\node[pin={[pin distance=0.1cm] -90:{\scriptsize $\textcolor{blue}{\pgfmathprintnumber[fixed, precision=3]{\slopeK}}$}},fill=blue,circle,scale=0.1,]
	(aninnernode) at (8,-4.5) {};

	\addplot[color=magenta,mark=+,]  table{Num/Num_P_zPcbc_4_10_expj9.txt};
	\pgfplotstableread{Num/Num_P_zPcbc_4_10_expj9.txt}\tableL
	\pgfplotstablecreatecol[linear regression]{regression}{\tableL}
	\xdef\slopeL{\pgfplotstableregressiona}
	 \node[pin={[pin distance=0.1cm] 90:{\scriptsize $\textcolor{magenta}{\pgfmathprintnumber[fixed, precision=3]{\slopeL}}$}},fill=magenta,circle,scale=0.1,]
	(aninnernode) at (9.2,-4.4) {};
	\node[ pin={ [pin distance=0.1cm] 90:{\scriptsize $\alpha = 2$ }}, fill=blue,circle,scale=0.1,] (aninnernode) at (7,-3.1) {};

	\nextgroupplot[xmode = log, ymode=log, grid = major, title = {$\gamma_j = 1/\pi^{2 \alpha}$}]
	
	\addplot[color=blue,mark=o,  dashed, mark options=solid]  table {Num/Num_S_zScbc_2_10_unweighted.txt};
	\pgfplotstableread{Num/Num_S_zScbc_2_10_unweighted.txt}\tableM
	\pgfplotstablecreatecol[linear regression]{regression}{\tableM}
	\xdef\slopeM{\pgfplotstableregressiona}
	 \node[pin={[pin distance=0.1cm] -90:{\scriptsize $\textcolor{blue}{\pgfmathprintnumber[fixed, precision=3]{\slopeM}}$}},fill=blue,circle,scale=0.1,]
	(aninnernode) at (8,-0.25) {};

	\addplot[color=magenta,mark=+,]  table {Num/Num_P_zPcbc_2_10_unweighted.txt};
	\pgfplotstableread{Num/Num_P_zPcbc_2_10_unweighted.txt}\tableN
	\pgfplotstablecreatecol[linear regression]{regression}{\tableN}
	\xdef\slopeN{\pgfplotstableregressiona}
	\node[pin={[pin distance=0.1cm] -90:{\scriptsize $\textcolor{magenta}{\pgfmathprintnumber[fixed, precision=3]{\slopeN}}$}},fill=magenta,circle,scale=0.1,]
	(aninnernode) at (9.2,1.2) {};
	\node[ pin={ [pin distance=0.1cm] -90:{\scriptsize $\alpha = 1$ }}, fill=blue,circle,scale=0.1,] (aninnernode) at (9.5,-1) {};

	\addplot[color=blue, mark=o,  dashed, mark options=solid]  table{Num/Num_S_zScbc_4_10_unweighted.txt};
	\pgfplotstableread{Num/Num_S_zScbc_4_10_unweighted.txt}\tableO
	\pgfplotstablecreatecol[linear regression]{regression}{\tableO}
	\xdef\slopeO{\pgfplotstableregressiona}
	\node[pin={[pin distance=0.1cm] -90:{\scriptsize $\textcolor{blue}{\pgfmathprintnumber[fixed, precision=3]{\slopeO}}$}},fill=blue,circle,scale=0.1,]
	(aninnernode) at (8, -4) {};

	\addplot[color=magenta,mark=+,]  table{Num/Num_P_zPcbc_4_10_unweighted.txt};
	\pgfplotstableread{Num/Num_P_zPcbc_4_10_unweighted.txt}\tableP
	\pgfplotstablecreatecol[linear regression]{regression}{\tableP}
	\xdef\slopeP{\pgfplotstableregressiona}
	\node[pin={[pin distance=0.1cm] 90:{\scriptsize $\textcolor{magenta}{\pgfmathprintnumber[fixed, precision=3]{\slopeP}}$}},fill=magenta,circle,scale=0.1,]
	(aninnernode) at (9.2,-3.8) {};
	\node[ pin={ [pin distance=0.1cm] 90:{\scriptsize $\alpha = 2$ }}, fill=blue,circle,scale=0.1,] (aninnernode) at (7,-2.5) {};
    \end{groupplot}

    \end{tikzpicture}
    \caption{The values of $\calS_n^{*} (\bsz_{\rm cbc}^{\calS})$ and $\calP_n^* (\bsz_{\rm cbc}^{\calP})$ 
    with $\bsz_{\rm cbc}^{\calS}$ and $\bsz_{\rm cbc}^{\calP}$ obtained from their respective CBC constructions 
    for different weight parameters for each $n = 2^{10}, 2^{11}, \ldots, 2^{14}$.
    The dimension is set to be $d=10$. 
    The negative numbers in each subplot denote the slopes of lines respectively.}
    \label{MoSrfig:pgf_num_d10}
\end{figure}
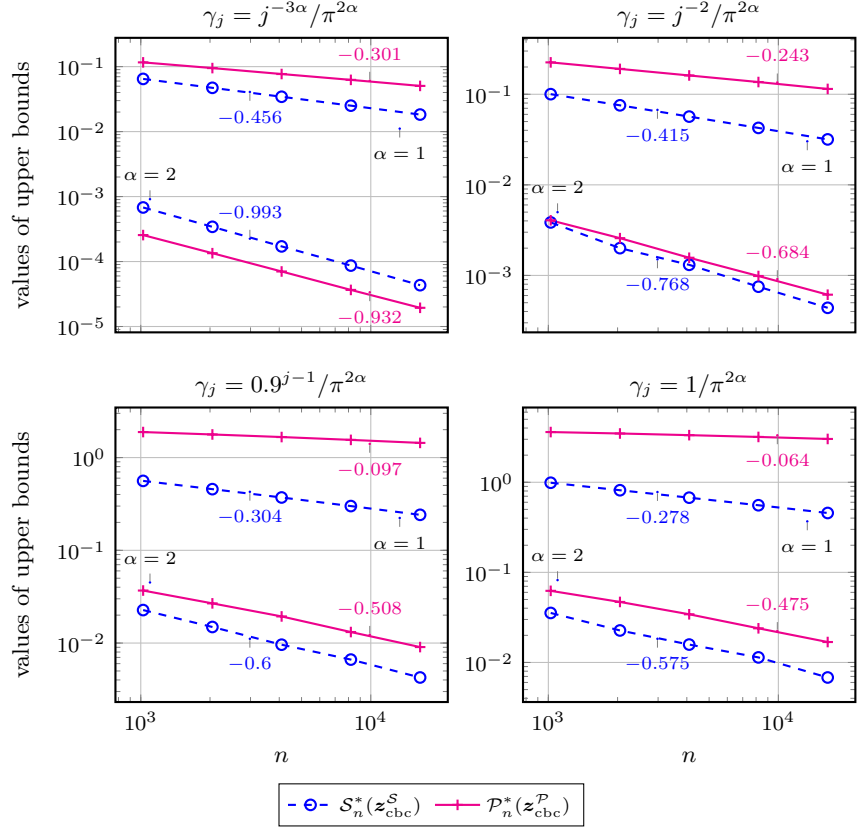

Figure~\ref{MoSrfig:pgf_num_d10} plots the values of the upper bounds against the number of points~$n = 2^m$ with $m=10,\ldots,14$ for dimension $d = 10$. We find that $\calS_n^{*}(\bsz)$ outperforms $\calP_n^*(\bsz)$ in almost all selected scenarios when $d = 10$, except for the case when $\gamma_j = j^{-3 \alpha}/\pi^{2\alpha}$ for $\alpha=2$. The negative numbers on each subplot of  Figure~\ref{MoSrfig:pgf_num_d10} indicate the convergence rates $\nu$ in $\calO(n^{-\nu + \delta})$ for $\delta>0$. We observe that $\calS_n^{*}(\bsz)$ exhibits the trend proven in \cite{CKNS20, KMN22} with $\nu$ close to $\alpha/2$, especially for fast decaying weights. However, for $\alpha = 1$, the convergence rates of $\calP_n^*(\bsz)$ perform worse than $\calS_n^{*}(\bsz)$, especially for slow decaying weights $\gamma_j = 0.9^{j-1} / \pi^{2 \alpha}$ and equal weights~$\gamma_j = 1 / \pi^{2 \alpha}$.

\begin{figure}[t]
\centering
\begin{tikzpicture}
 \begin{customlegend}[legend columns=4, legend style={at={(4.2,-6)},anchor=north, align=center},,
            legend entries={$\calS_n^{*} (\bsz_{\rm cbc}^{\calS})$ ,
                            $\calP_n^* (\bsz_{\rm cbc}^{\calP})$,
                            }]
        \addlegendimage{color= blue, dashed, mark options=solid, line width=0.8 pt,   line legend}
            \addlegendimage{color= magenta, line width=0.8 pt,   line legend} 
    \end{customlegend}
    
    \begin{groupplot}[
            group style={
        group name=my plots, 
        group size=2 by 2,
        x descriptions at=edge bottom,
        ylabels at=edge left,
        },
                footnotesize,
                width=6cm,
                height=5.5cm,
                xlabel= {$n$},
                ylabel= {values of upper bounds},
            ]

\nextgroupplot[xmode = log, ymode=log, grid = major, title = { $\gamma_j = j^{-3 \alpha} / \pi^{2 \alpha}$}]

\addplot[color=blue,  dashed,]  table {Dim/Dim_S_zScbc_2_100_invj3.txt};

\addplot[color=magenta,]  table {Dim/Dim_P_zPcbc_2_100_invj3.txt};
\node[ pin={ [pin distance=0.1cm] -90:{\scriptsize $\alpha = 1$ }}, fill=blue,circle,scale=0.1,] (aninnernode) at (4,-3.3) {};

\addplot[color=blue, dashed, ]  table{Dim/Dim_S_zScbc_4_100_invj6.txt};

\addplot[color=magenta,]  table{Dim/Dim_P_zPcbc_4_100_invj6.txt};
\node[ pin={ [pin distance=0.1cm] -90:{\scriptsize $\alpha = 2$ }}, fill=blue,circle,scale=0.1,] (aninnernode) at (4,-8.7) {};

\nextgroupplot[xmode = log, ymode=log, grid = major, title = { $\gamma_j =  j^{-2}/\pi^{2 \alpha}$}]
\addplot[color=blue, dashed,]  table {Dim/Dim_S_zScbc_2_100_invj2.txt};

\addplot[color=magenta,]  table {Dim/Dim_P_zPcbc_2_100_invj2.txt};
\node[ pin={ [pin distance=0.1cm] -90:{\scriptsize $\alpha = 1$ }}, fill=blue,circle,scale=0.1,] (aninnernode) at (4,-2.5) {};

\addplot[color=blue,  dashed]  table{Dim/Dim_S_zScbc_4_100_invj2.txt};

\addplot[color=magenta,]  table{Dim/Dim_P_zPcbc_4_100_invj2.txt};
\node[ pin={ [pin distance=0.1cm] -90:{\scriptsize $\alpha = 2$ }}, fill=blue,circle,scale=0.1,] (aninnernode) at (4,-5.5) {};

\nextgroupplot[xmode = log, ymode=log, grid = major, title = { $\gamma_j = 0.9^{j-1}/\pi^{2 \alpha}$}]
\addplot[color=blue, dashed,]  table {Dim/Dim_S_zScbc_2_100_expj9.txt};

\addplot[color=magenta,]  table {Dim/Dim_P_zPcbc_2_100_expj9.txt};
\node[ pin={ [pin distance=0.1cm] 90:{\scriptsize $\alpha = 1$ }}, fill=blue,circle,scale=0.1,] (aninnernode) at (0.5,-1.2) {};

\addplot[color=blue,  dashed, ]  table{Dim/Dim_S_zScbc_4_100_expj9.txt};

\addplot[color=magenta,]  table{Dim/Dim_P_zPcbc_4_100_expj9.txt};
\node[ pin={ [pin distance=0.1cm] -90:{\scriptsize $\alpha = 2$ }}, fill=blue,circle,scale=0.1,] (aninnernode) at (1,-8) {};

\nextgroupplot[xmode = log, ymin=1e-6, ymax = 1e7,ytick={1e-5,1e-3,1e-1,1e1,1e3,1e5,1e7}, ymode=log, grid = major, title = {$\gamma_j = 1/\pi^{2 \alpha}$}]
\addplot[color=blue,  dashed, ]  table {Dim/Dim_S_zScbc_2_100_unweighted.txt};

\addplot[color=magenta,]  table {Dim/Dim_P_zPcbc_2_100_unweighted.txt};
\node[ pin={ [pin distance=0.1cm] 90:{\scriptsize $\alpha = 1$ }}, fill=blue,circle,scale=0.1,] (aninnernode) at (0.5,-1.2) {};

\addplot[color=blue,  dashed,]  table{Dim/Dim_S_zScbc_4_100_unweighted.txt};

\addplot[color=magenta,]  table{Dim/Dim_P_zPcbc_4_100_unweighted.txt};
\node[ pin={ [pin distance=0.1cm] -90:{\scriptsize $\alpha = 2$ }}, fill=blue,circle,scale=0.1,] (aninnernode) at (1,-8) {};

\end{groupplot}

\end{tikzpicture}
\caption{The values of $\calS_n^{*} (\bsz_{\rm cbc}^{\calS})$ and $\calP_n^* (\bsz_{\rm cbc}^{\calP})$  with generating vector obtained
by respective CBC construction against $d$ with $d = 1, 2,\ldots, 100$ for different weight parameters.
The number of points is fixed as $n = 2^{10}$. }

\label{MoSrfig:pgf_dim}
\end{figure}
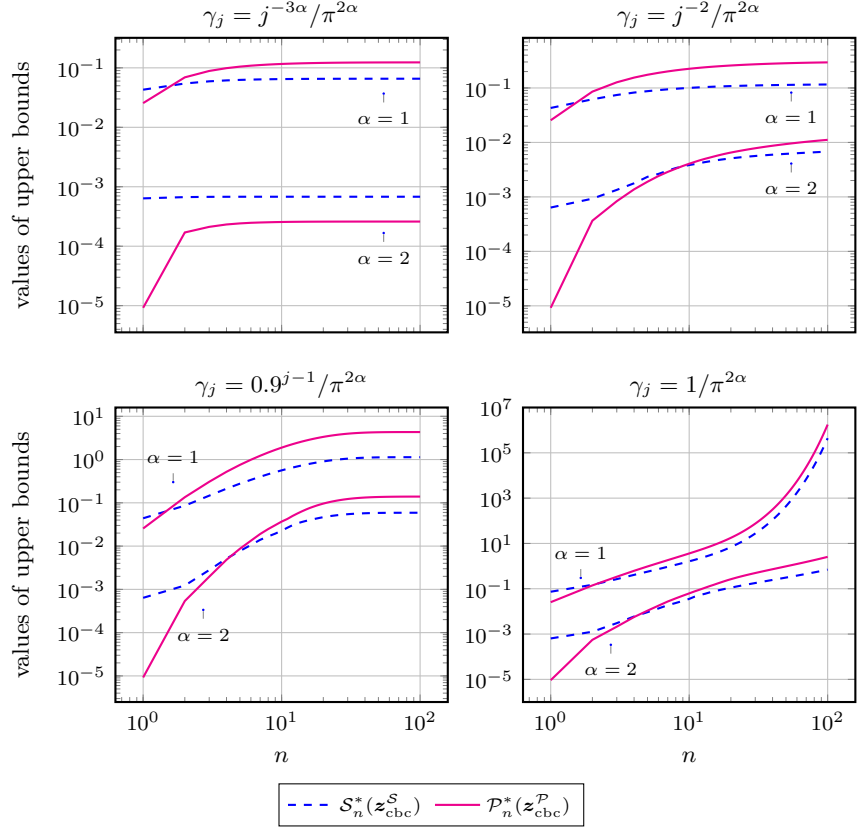

To observe how dimension may affect the performance of each metric, we plot Figure~\ref{MoSrfig:pgf_dim}. This plot shows the two upper bounds for fixed $n = 2^{10}$ as dimension increases from the first dimension till $d=100$. We see that although Figure~\ref{MoSrfig:pgf_num_d10} indicates $\calS_n^{*}(\bsz)$ performs better than $\calP_n^*(\bsz)$ in most cases, this is not the case for smaller dimensions where $\calP_n^{*}(\bsz)$ may outperform the $\calS_n^{*}(\bsz)$ criterion. For $d=10$ and $n=2^{10}$, Figure~\ref{MoSrfig:pgf_dim} is consistent with the findings of Figure~\ref{MoSrfig:pgf_num_d10}.

We plot Figure \ref{MoSrfig:pgf_comp_d10} to observe how $\calS_n^{*}(\bsz)$ performs with the generating vector $\bsz_{\rm cbc}^{\calP}$ constructed by Algorithm~\ref{MoSralg:CBC_ker_fft} and how $\calP_n^{*}(\bsz)$ performs with the generating vector $\bsz_{\rm cbc}^{\calS}$ constructed by Algorithm~\ref{MoSralg:CBC_four_fft}. 
The plot directly compares $\calP_n^{*}(\bsz)$ evaluated at both $\bsz_{\rm cbc}^{\calS}$ and $\bsz_{\rm cbc}^{\calP}$. This is also done for $\calS_n^{*}(\bsz)$.
We can see that $\calP_n^{*} (\bsz_{\rm cbc}^{\calS})$ is very similar to $\calP_n^{*} (\bsz_{\rm cbc}^{\calP})$ in all selected scenarios, while $\calS_n^{*} (\bsz_{\rm cbc}^{\calP})$ is close to $\calS_n^{*} (\bsz_{\rm cbc}^{\calS})$ for $\alpha=1$ but not as close for $\alpha=2$ with fast decaying weights. We find that even if we construct the lattice using $\calS_n^*(\bsz)$, the error when measured against the $\calP_n^*(\bsz)$ criterion is small.

\pgfplotsset{ tick label style={font=\small}, label style={font=\small},
legend style={font=\scriptsize}, every axis/.append style={ line
width=0.8 pt, tick style={line width=0.3pt}} }
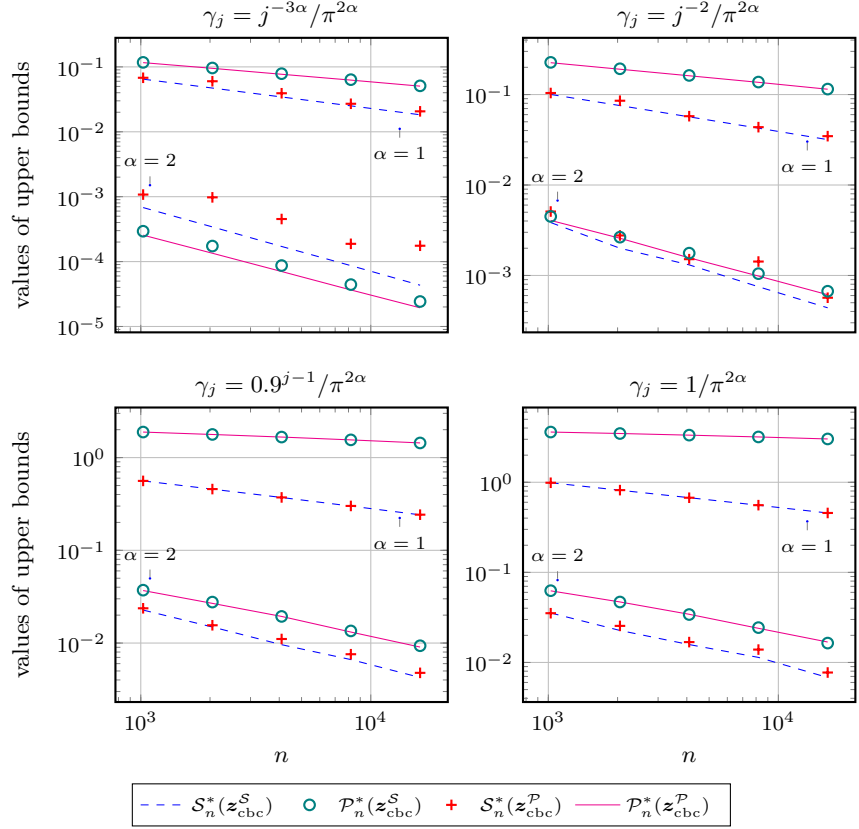
\begin{figure}[t]
    \centering
    \begin{tikzpicture}

    \begin{customlegend}[legend columns=4, legend style={at={(4.2,-6)},anchor=north, align=center},,
            legend entries={$\calS_n^{*} (\bsz_{\rm cbc}^{\calS})$  ,
        				$\calP_n^*(\bsz_{\rm cbc}^{\calS})$,    
				$\calS_n^{*} (\bsz_{\rm cbc}^{\calP})$ , 
				$\calP_n^* (\bsz_{\rm cbc}^{\calP})$ ,         
		              }]
        \addlegendimage{color= blue, dashed, mark options=solid, line width=0.3 pt,   line legend}
            \addlegendimage{color= white, mark options={color=teal}, mark=o, line width= 0.8pt, line legend} 
            \addlegendimage{color= white, mark=+,  mark options={color=red, mark size=2},line width=0.8pt,  line legend}
            \addlegendimage{color= magenta, line width = 0.3pt, line legend }
    \end{customlegend}
    
    \begin{groupplot}[
            group style={ 
        group name=my plots, 
        group size=2 by 2,
        x descriptions at=edge bottom,
        ylabels at=edge left,
        },
                footnotesize,
                width=6cm,
                height=5.5cm,
                xlabel= {$n$},
              ylabel= {values of upper bounds},
            ]       
       \nextgroupplot[xmode = log, ymode=log, grid = major, title = { $\gamma_j = j^{-3 \alpha} / \pi^{2 \alpha}$}]

	\addplot[color=blue , dashed, mark options=solid, thin] table{Num/Num_S_zScbc_2_10_invj3.txt};

	\addplot[color=red, mark=+, dashed, mark options=solid,only marks] table{Num/Num_S_zPcbc_2_10_invj3.txt};

	\addplot[color=magenta,thin]  table {Num/Num_P_zPcbc_2_10_invj3.txt};
		
	\addplot[color= teal , mark=o, only marks]  table {Num/Num_P_zScbc_2_10_invj3.txt}; 
	\node[ pin={ [pin distance=0.1cm] -90:{\scriptsize $\alpha = 1$ }}, fill=blue,circle,scale=0.1,] (aninnernode) at (9.5,-4.5) {};
 
	\addplot[color=blue, dashed, mark options=solid, thin]  table{Num/Num_S_zScbc_4_10_invj6.txt};

	\addplot[color=red, mark=+, dashed, mark options=solid ,only marks]  table{Num/Num_S_zPcbc_4_10_invj6.txt};

	\addplot[color=magenta, thin]  table{Num/Num_P_zPcbc_4_10_invj6.txt};
	
	\addplot[color= teal  , mark=o, only marks]  table{Num/Num_P_zScbc_4_10_invj6.txt};
	\node[ pin={ [pin distance=0.1cm] 90:{\scriptsize $\alpha = 2$ }}, fill=blue,circle,scale=0.1,] (aninnernode) at (7,-6.5) {};

	\nextgroupplot[xmode = log, ymode=log, grid = major, title = { $\gamma_j =  j^{-2}/\pi^{2 \alpha}$}]
	
	\addplot[color=blue,  dashed, mark options=solid,thin]  table {Num/Num_S_zScbc_2_10_invj2.txt};
	
	\addplot[color=red, mark=+,  dashed, mark options=solid,only marks]  table {Num/Num_S_zPcbc_2_10_invj2.txt};

	\addplot[color=magenta, thin]  table {Num/Num_P_zPcbc_2_10_invj2.txt};
	
	\addplot[color= teal ,mark=o,only marks]  table {Num/Num_P_zScbc_2_10_invj2.txt};
	\node[ pin={ [pin distance=0.1cm] -90:{\scriptsize $\alpha = 1$ }}, fill=blue,circle,scale=0.1,] (aninnernode) at (9.5,-3.5) {};

	\addplot[color=blue,  dashed, mark options=solid,thin]  table{Num/Num_S_zScbc_4_10_invj2.txt};
	
	\addplot[color=red, mark=+,  dashed, mark options=solid,only marks]  table{Num/Num_S_zPcbc_4_10_invj2.txt};

	\addplot[color=magenta, thin]  table{Num/Num_P_zPcbc_4_10_invj2.txt};
	
	\addplot[color= teal ,mark=o,only marks]  table{Num/Num_P_zScbc_4_10_invj2.txt};
	\node[ pin={ [pin distance=0.1cm] 90:{\scriptsize $\alpha = 2$ }}, fill=blue,circle,scale=0.1,] (aninnernode) at (7,-5) {};

	\nextgroupplot[xmode = log, ymode=log, grid = major, title = { $\gamma_j = 0.9^{j-1}/\pi^{2 \alpha}$}]
	
	\addplot[color=blue,   dashed, mark options=solid,thin]  table {Num/Num_S_zScbc_2_10_expj9.txt};
	
	\addplot[color=red, mark=+,  dashed, mark options=solid,only marks]  table {Num/Num_S_zPcbc_2_10_expj9.txt};

	\addplot[color=magenta,thin]  table {Num/Num_P_zPcbc_2_10_expj9.txt};

	\addplot[color= teal  ,mark=o,only marks]  table {Num/Num_P_zScbc_2_10_expj9.txt};
	\node[ pin={ [pin distance=0.1cm] -90:{\scriptsize $\alpha = 1$ }}, fill=blue,circle,scale=0.1,] (aninnernode) at (9.5,-1.5) {};

	\addplot[color=blue,  dashed, mark options=solid,thin]  table{Num/Num_S_zScbc_4_10_expj9.txt};
	
	\addplot[color=red, mark=+,  dashed, mark options=solid,only marks]  table{Num/Num_S_zPcbc_4_10_expj9.txt};

	\addplot[color=magenta,thin]  table{Num/Num_P_zPcbc_4_10_expj9.txt};
	
	\addplot[color= teal  ,mark=o, only marks]  table{Num/Num_P_zScbc_4_10_expj9.txt};
	\node[ pin={ [pin distance=0.1cm] 90:{\scriptsize $\alpha = 2$ }}, fill=blue,circle,scale=0.1,] (aninnernode) at (7,-3) {};

	\nextgroupplot[xmode = log, ymode=log, grid = major, title = {$\gamma_j = 1/\pi^{2 \alpha}$}]
	
	\addplot[color=blue,   dashed, mark options=solid,thin]  table {Num/Num_S_zScbc_2_10_unweighted.txt};

	\addplot[color=red,mark=+,  dashed, mark options=solid,only marks]  table {Num/Num_S_zPcbc_2_10_unweighted.txt};

	\addplot[color=magenta,thin]  table {Num/Num_P_zPcbc_2_10_unweighted.txt};

	\addplot[color= teal  ,mark=o,only marks]  table {Num/Num_P_zScbc_2_10_unweighted.txt};
	\node[ pin={ [pin distance=0.1cm] -90:{\scriptsize $\alpha = 1$ }}, fill=blue,circle,scale=0.1,] (aninnernode) at (9.5,-1) {};

	\addplot[color=blue, dashed,thin]  table{Num/Num_S_zScbc_4_10_unweighted.txt};

	\addplot[color=red, mark=+,  dashed, mark options=solid,only marks]  table{Num/Num_S_zPcbc_4_10_unweighted.txt};

	\addplot[color=magenta,thin]  table{Num/Num_P_zPcbc_4_10_unweighted.txt};

	\addplot[color= teal  ,mark=o,  only marks]  table{Num/Num_P_zScbc_4_10_unweighted.txt};
	\node[ pin={ [pin distance=0.1cm] 90:{\scriptsize $\alpha = 2$ }}, fill=blue,circle,scale=0.1,] (aninnernode) at (7,-2.5) {};
	
    \end{groupplot}

    \end{tikzpicture}
    \caption{The values of $\calS_n^{*} (\bsz_{\rm cbc}^{\calS})$, $\calP_n^{*} (\bsz_{\rm cbc}^{\calS})$, $\calS_n^{*} (\bsz_{\rm cbc}^{\calP})$,
and $\calP_n^* (\bsz_{\rm cbc}^{\calP})$  with $\bsz_{\rm cbc}^{\calS}$ and $\bsz_{\rm cbc}^{\calP}$ obtained
    from their respective CBC constructions for different weight parameters for $n = 2^{10}, 2^{11}, \ldots, 2^{14}$.
    The dimension is set to be $d=10$.}
    \label{MoSrfig:pgf_comp_d10}
\end{figure}

The overall computation cost of fast CBC construction using $\calS_n^{*}(\bsz)$ as the search criterion for approximation with product weights is $\calO(d n \log n)$~\cite{CKNS21}, while the computation cost of CBC construction using $\calP_n^*(\bsz)$ is $\calO(d n^2 \log n)$, thus it is more efficient to construct a lattice using $\calS^*_n(\bsz)$.

For large $n$ (especially if $\alpha>1$)  the entries of the matrix~$\calM$, appearing in the calculation of $\calP^*_n(\bsz)$, start to become very close to 1, resulting in an ill-conditioned matrix. Thus, double precision does not provide a sufficient level of accuracy, and computation in arbitrary precision becomes necessary to accurately compute eigenvalues, leading to increased computation time. Such issues only occur for much larger $n$ when computing with $\calS_n^{*}(\bsz)$.

We conclude that $\calS_n^{*}(\bsz)$ is the more useful choice for generating lattice-based algorithms for multivariate approximation. From an implementation point of view, we see that computing with $\calS_n^{*}(\bsz)$ is more efficient. We can make use of the fast CBC Algorithm \ref{MoSralg:CBC_four_fft} for computing $\calS^*_n(\bsz)$ which has a smaller computational cost compared to Algorithm \ref{MoSralg:CBC_ker_fft} using $\calP_n^{*}(\bsz)$ as the search criterion. In addition, $\calS_n^{*}(\bsz)$ can be computed using double-precision for choices of $n$ where arbitrary precision is necessary for computing $\calP_n^{*}(\bsz)$. From a numerical point of view, we also see that for the cases tested there is only a small difference in the worst-case error upper bound  for the lattice generated using $\calS_n^{*}(\bsz)$ even when measured by $\calP_n^{*}(\bsz)$.
\newline

\textbf{Acknowledgements}\quad We would like to acknowledge the support from the Australian Research Council (DP21010083) and the Research Foundation Flanders (FWO G091920N).  We would also like to extend our thanks to Vesa Kaarnioja and Ronald Cools for their insights.


\begin{thebibliography}{99.}

\bibitem{BBC20}
{Belhadji, A., Bardenet, R., Chainais., P.}:
 {Kernel interpolation with continuous volume sampling.}
 In: Proceedings of the 37th International Conference on Machine Learning (ICML'20).
 JMLR.org, Article 68, 725--735 (2020).

 \bibitem{BKUV17}
{Byrenheid, G., K\"ammerer, L., Ullrich, T., Volkmer, T.:}
{Tight error bounds for rank-1 lattice sampling in spaces of hybrid mixed smoothness.}
Numer. Math. \textbf{136}, 993--1034 (2017)

\bibitem{CKNS20}
{Cools, R., Kuo, F.Y., Nuyens, D., Sloan, I.H.:}
{Lattice algorithms for multivariate approximation in periodic spaces with general weights.}
Contemp. Math. \textbf{754}, 93--113 (2020)

\bibitem{CKNS21}
{Cools, R., Kuo, F.Y., Nuyens, D., Sloan, I.H.:}
{Fast CBC construction of lattice algorithms for multivariate approximation with POD and SPOD weights.}
Math. Comput. \textbf{90}, 787--812 (2021)

\bibitem{CN08}
 {Cools, R., Nuyens, D.}:
 {A Belgian view on lattice rules.}
 In: Keller, A., Heinrich, S., Niederreiter,  H. (eds.) Monte Carlo and Quasi-Monte Carlo Methods 2006,
 pp. 3--21. Springer (2008)

 \bibitem{DSW03}
 {De Marchi, S., Schaback, R., Wendland, H.}:
 {Near-optimal data-independent point locations for radial basis function interpolation}.
 Adv. Comput. Math. \textbf{23}, 317--330 (2005)


\bibitem{DKKS07}
{ Dick, J., Kritzer, P., Kuo. F.Y., and Sloan, I.H.:}
{Lattice-Nystrom method for Fredholm integral equations of the second kind with convolution type kernels.}
J. Complexity, \textbf{23}, 752 -- 772 (2007).

 \bibitem{DKS13}
 {Dick, J., Kuo, F.Y., Sloan, I.H.:}
 {High-dimensional integration: the Quasi-Monte Carlo way.}
 Acta Numer. \textbf{22}, 133--288 (2013)

\bibitem{DKU23} 
	{Dolbeault, M. Krieg, D., Ullrich, M.:}
	{A sharp upper bound for sampling numbers in $L_2$.}
	Appl. Comput. Harmon. Anal. \textbf{63}, 113--134 (2023).

 \bibitem{GIKV21}
 {Gross, C., Iwen,  M.A., K\"ammerer, L., Volkmer, T.:}
 {A deterministic algorithm for constructing multiple rank-$1$ lattices of near-optimal size.}
 Adv. Comput. Math. \textbf{47}, 86 (2021)

 \bibitem{KKKN20}
 {Kaarnioja, V., Kazashi, Y., Kuo, F.Y., Nobile, F., Sloan, I.H.}:
 { Fast approximation by periodic kernel-based lattice-point interpolation with application in uncertainty quantification}.
 Numer.~Math. \textbf{150}, 33--77 (2022)
 
 \bibitem{KV19}
 { K\"ammerer, L., Volkmer, T.:}
 {Approximation of multivariate periodic functions based on sampling along multiple rank-$1$ lattices.}
 J. Approx. Theory \textbf{246}, 1--27 (2019)

 \bibitem{KMN22}
 {Kuo, F.Y., Mo, W., Nuyens, D.}:
 {Constructing embedded lattice-based algorithms for multivariate function approximation with a composite number of points}.
 {doi:10.48550/ARXIV.2209.01002}

 \bibitem{KSW06}
 {Kuo, F.Y., Sloan, I.H., Wo\'zniakowski, H.}:
 {Lattice rules for multivariate approximation in the worst case setting}.
 In: Niederreiter, H., Talay, D. (eds.) Monte Carlo and Quasi-Monte Carlo Methods 2004,
 pp. 289--330. Springer (2006)



 \bibitem{LM12}
 {L'Ecuyer, P., Munger, D.:}
 {On figures of merit for randomly shifted lattice rules.}
 In: Plaskota, L., Wo\'zniakowski, H. (eds.) Monte Carlo and Quasi-Monte Carlo Methods 2010,
 pp. 133--159. Springer (2012)

 \bibitem{Nuy14}
 {Nuyens, D.:}
 {The construction of good lattice rules and polynomial lattice rules.}
 In: Kritzer, P., Niederreiter, H., Pillichshammer, F., Winterhof, A. (eds.)
Uniform Distribution and Quasi-Monte Carlo Methods.
Radon Series on Computational and Applied Mathematics Vol. 15,
pp. 223--256, De Gruyter (2014)

\bibitem{NC06a}
{Nuyens, D., Cools, R.:}
{Fast algorithms for component-by-component construction of rank-$1$ lattice rules in shift-invariant reproducing kernel Hilbert spaces.}
Math. Comp. \textbf{75}, 903--920 (2006)

\bibitem{NC06b}
{Nuyens, D., Cools, R.:}
{Fast component-by-component construction of rank-$1$ lattice rules with a non-prime number of points.}
J. Complexity \textbf{22}, 4--28 (2006)

 \bibitem{NSW04}
 {Novak, E., Sloan, I.H., Wo\'zniakowski, H.:}
 {Tractability of approximation for weighted Korobov spaces on classical and quantum computers.}
 Found. Comput. Math. \textbf{4}, 121--156 (2004)

 \bibitem{DLMF}
 {Olver, F.W.J., et al.:}
 {NIST Digital Library of Mathematical Functions.}
  http://dlmf.nist.gov/ (2022). Accessed 07 Dec 2022

  \bibitem{SJ94}
  {Sloan, I.H., Joe, S.:}
  {Lattice methods for multiple integration.}
  Oxford University Press, Oxford, 1994

  \bibitem{S95}
  {Schaback, R.}:
  {Error estimates and condition numbers for radial basis function interpolation.}
  Adv.~Comput.~Math. \textbf{3}, 251--264 (1995)

  \bibitem{SW06}
  {Schaback, R., Wendland, H.}:
  {Kernel techniques: From machine learning to meshless methods.}
  Acta Numerica \textbf{15},  543--639 (2006)


 \bibitem{WS93}
 {Wu, Z.M., Schaback, R.}:
 {Local error estimates for radial basis function interpolation of scattered data}.
 IMA J.~Numer.~Anal. \textbf{13}, 13--27 (1993)

 \bibitem{ZKH09}
 {Zeng, X.Y., Kritzer, P., Hickernell,  F.J.}:
 {Spline methods using integration lattices and digital nets}.
 Constr. Approx. \textbf{30}, 529--555 (2009)

 \bibitem{ZLH06}
 {Zeng, X.Y., Leung, K.T., Hickernell,  F.J.}:
 {Error analysis of splines for periodic problems using lattice designs}.
 In: Niederreiter, H., Talay, D. (eds.) Monte Carlo and Quasi-Monte Carlo Methods 2004,
 pp. 501--514. Springer (2006)

\end{thebibliography}
\end{document}